\documentclass[12pt]{article}
\usepackage{amsmath}
\usepackage{epsfig}
\usepackage{amsthm}
\usepackage{amsfonts}
\usepackage{color}
\usepackage{verbatim}
\usepackage{rotating}
\font\tenopen=msbm10
\font\sevenopen=msbm7
\font\fiveopen=msbm5
\newfam\openfam
\def\open{\fam\openfam\tenopen}
\textfont\openfam=\tenopen
\scriptfont\openfam=\sevenopen
\scriptscriptfont\openfam=\fiveopen
\font\title=cmbx10 scaled\magstep1

\def\N{{\open \mathbb{N}}}
\def\Z{{\open \mathbb{Z}}}

\def\R{{\open \mathbb{R}}}

\def\p{{\open \mathbb{P}}}
\def\E{{\open \mathbb{E}}}

\def\B{I\negthinspace\negthinspace B}

\theoremstyle{remark} \theoremstyle{lemma} \theoremstyle{definition}
\theoremstyle{corol} \theoremstyle{proposition}
\theoremstyle{condition} \theoremstyle{conjecture}
\newtheorem{theorem}{\bf{Theorem}}

\newtheorem{lemma}{\bf{Lemma}}

\newtheorem{corol}{\bf{Corollary}}

\newtheorem{condition}{\bf{Condition}}

\begin{document}

\centerline{{\bf{\large Block Sampling under Strong Dependence}}
\footnote{{\em JEL Codes:} C22\\
\phantom{Ke}{\em Keywords:} Asymptotic normality; Covariance; Linear processes; Long-range dependence; Rosenblatt distribution; Hermite processes}}

\centerline{\sc By Ting Zhang$^\dag$, Hwai-Chung Ho$^\ddag$, Martin Wendler$^\dag$ and Wei Biao Wu$^\dag$}
\centerline{\em $^\dag$The University of Chicago and $^\ddag$Institute of Statistical Science, Academia Sinica, Taipei}
\centerline{\today}

\begin{abstract}
The paper considers the block sampling method for long-range dependent processes. Our theory generalizes earlier ones by Hall, Jing and Lahiri (1998) on functionals of Gaussian processes and Nordman and Lahiri (2005) on linear processes. In particular, we allow nonlinear transforms of linear processes. Under suitable conditions on physical dependence measures, we prove the validity of the block sampling method. The problem of estimating the self-similar index is also studied.
\end{abstract}

\section{Introduction}
Long memory (strongly dependent, or long-range dependent)
processes have received considerable attention in areas including
econometrics, finance, geology and telecommunication among others.
Let $X_i$, $i \in \Z$, be a stationary linear process of the form
\begin{equation}\label{eqn:Xidef}
X_i = \sum_{j=0}^\infty a_j \varepsilon_{i-j},
\end{equation}
where $\varepsilon_i$, $i \in \Z$, are independent and identically
distributed (iid) random variables with zero mean, finite variance
and $(a_j)_{j=0}^\infty$ are square summable real coefficients. If
$a_i \to 0$ very slowly, say $a_i \sim i^{-\beta}$, $1/2 < \beta < 1$,
then there exists a constant $c_\beta > 0$ such that the covariances
$\gamma_i = \E(X_0 X_i) = \E(\varepsilon_0^2) \sum_{j=0}^\infty a_j a_{i+j}
\sim c_\beta \E(\varepsilon_0^2) i^{1-2\beta}$ are not summable, thus suggesting strong dependence. An
important example is the fractionally integrated autoregressive
moving average (FARIMA) processes (Granger and Joyeux, 1980 and Hosking, 1981). Let $K$ be a measurable function such that $\E [K^2(X_i)] < \infty$, and $\mu = \E K(X_i)$. This paper considers the asymptotic sampling distribution of
\begin{equation*}
\hat \mu_n = {1\over n} \sum_{i=1}^n K(X_i) = { {S_n} \over n } + \mu,\ \mathrm{where}\ S_n = \sum_{i=1}^n [K(X_i) - \mu].
\end{equation*}
In the inference of the mean $\mu$, such as the construction of
confidence intervals and hypothesis testing, it is
necessary to develop a large sample theory for the partial sum
process $S_n$. The latter problem has a substantial history. Here
we shall only give a very brief account. Davydov (1970) considered
the special case $K(x) = x$ and Taqqu (1975) and Dobrushin and
Major (1979) dealt with another special case in which $K$ can be a
nonlinear transform while $(X_i)$ is a Gaussian process. Quadratic
forms are considered in Chung (2002). See Surgailis (1982), Avram
and Taqqu (1987) and Dittmann and Granger (2002) for other
contributions and Wu (2006) for further references. For general
linear processes with nonlinear transforms, under some regularity
conditions on $K$, if $X_i$ is a short memory (or short-range
dependent) process with $\sum_{j=0}^\infty |a_j| < \infty$, then
$S_n / \sqrt n$ satisfies a central limit theorem with a Gaussian
limiting distribution; if $X_i$ is long-memory (or long-range
dependent), then with proper normalization, $S_n$ may have either a
non-Gaussian or Gaussian limiting distribution and the normalizing
constant may no longer be $\sqrt n$ (Ho and Hsing, 1997 and Wu, 2006). In many situations, the non-Gaussian limiting distribution
can be expressed as a multiple Wiener-It\^{o} integral (MWI); see equation (\ref{eqn:MWI}).

The distribution function of a non-Gaussian WMI does not have a close form. This brings considerable
inconveniences in the related statistical inference. As a useful
alternative, we can resort to re-sampling techniques to estimate
the sampling distribution of $S_n$. K\"{u}nsch (1989) proved the
validity of the moving block bootstrap method for weakly dependent
stationary processes. However, Lahiri (1993) showed that, for
Gaussian subordinated long-memory processes, the block
bootstrapped sample means are always asymptotically Gaussian; thus
it fails to recover the non-Gaussian limiting distribution of the
multiple Wiener-It\^{o} integrals. On the other hand, Hall, Horowitz
and Jing (1995) proposed a sampling windows method. Hall, Jing and
Lahiri (1998) showed that, for the special class of processes of
nonlinear transforms of Gaussian processes, the latter method is
valid in the sense that the empirical distribution functions of
the consecutive block sums converge to the limiting distribution
of $S_n$ with a proper normalization. Nordman and Lahiri (2005)
proved that the same method works for linear processes, an
entirely different special class of stationary processes. However,
for linear processes, the limiting distribution is always
Gaussian. It has been an open problem whether a limit theory can
be established for a more general class of long-memory processes.

Here we shall provide an affirmative answer to the above question
by allowing functionals of linear processes, a more general class
of stationary processes which include linear processes and
nonlinear transforms of Gaussian processes as special cases.
Specifically, given a realization $Y_i = K(X_i)$, $1 \leq i \leq n$, with both $K$ and $X_i$ being possibly unknown or unobserved, we
consider consistent estimation of the sampling distribution of
$S_n / n$. To this end, we shall implement the concept of
physical dependence measures (Wu, 2005) which quantify the dependence of a random
process by measuring how outputs depend on inputs. The rest of the paper is organized as follows. Section
\ref{sec:mainsec} presents the main results and it deals with the
asymptotic consistency of the empirical distribution functions of
the normalized consecutive block sums. It is interesting to
observe that the same sampling windows method works for both
Gaussian and non-Gaussian limiting distributions. A simulation
study is provided in Section \ref{sec:simu}, and some proofs are
deferred to the Appendix.

\section{Main Results}\label{sec:mainsec}
In Section \ref{sec:asym}, we briefly review the asymptotic theory of $S_n$ in Ho and Hsing (1997) and Wu (2006). The block sampling method of Hall, Horowitz and Jing (1995) is described in Section \ref{subsec:blocksampling}. With physical dependence measures, Section \ref{sec:cesd} presents a consistency result for empirical sampling distributions. In Section \ref{subsec:sltilde}, we obtain a convergence rate for a variance estimate of $s_l^2 = \|S_l\|^2$. A consistent estimate of $H$, the self-similar parameter of the limiting process, is proposed in Section \ref{subsec:Hhat}.

For two positive sequences $(a_n)$ and $(b_n)$, write $a_n \sim
b_n$ if $a_n / b_n \to 1$ and $a_n \asymp b_n$ if there exists a
constant $C > 0$ such that $a_n / C \le b_n \le C a_n$ holds for
all large $n$. Let $\mathcal C_A$ (resp. $\mathcal C^p_A$) denote the collection
of continuous functions (resp. functions having $p$-th order continuous derivatives) on $A \subseteq \mathbb R$. Denote by ``$\Rightarrow$" the weak convergence; see Billingsley (1968) for a detailed account for the weak convergence theory on $\mathcal C_{[0,1]}$. For a random variable $Z$,
we write $Z \in {\cal L}^\nu$, $\nu > 0$, if $\| Z \|_\nu = (\E
|Z|^\nu )^{1/\nu} < \infty$, and write $\| Z \| = \| Z \|_2$. For
integers $i \le j$ define ${\cal F}_i^j = (\varepsilon_i,
\varepsilon_{i+1}, \ldots, \varepsilon_j)$. Write ${\cal F
}_i^\infty = (\varepsilon_i, \varepsilon_{i+1}, \ldots)$ and
${\cal F}_{-\infty}^j= (\ldots, \varepsilon_{j-1}, \varepsilon_j)$. Define the projection operator ${\cal P}_j$, $j \in \mathbb Z$, by
\begin{equation*}
{\cal P}_j \cdot = \E( \cdot | {\cal F}_{-\infty}^j) - \E( \cdot | {\cal F}_{-\infty}^{j-1}).
\end{equation*}
Then ${\cal P}_j \cdot$, $j \in \mathbb Z$, yield martingale differences.

\subsection{Asymptotic distributions}\label{sec:asym}
To study the asymptotic distribution of $S_n$
under strong dependence, we shall introduce the concept of power
rank (Ho and Hsing, 1997). Based on $K$ and $X_n$, let $X_{n, i} =
\sum_{j=n-i}^\infty a_j \varepsilon_{n-j} = \E(X_n | {\cal F
}_{-\infty}^i)$ be the tail process and define functions
\begin{equation*}
K_\infty(x) = \E K(x + X_n) \mbox{ and } K_n(x) = \E K(x + X_n - X_{n, 0}).
\end{equation*}
Note that $X_n-X_{n,0} = \sum_{j=0}^{n-1} a_j \varepsilon_{n-j}$
is independent of $X_{n,0}$. Denote by $\kappa_r =
K^{(r)}_\infty(0)$, the $r$-th derivative, if it exists. If $p \in
\N$ is such that $\kappa_p \not= 0$ and $\kappa_r =0$ for all
$r=1, \ldots, p-1$, then we say that $K$ has power rank $p$ with
respect to the distribution of $X_i$. The limiting distribution of
$S_n$ can be Gaussian or non-Gaussian. The non-Gaussian limiting
distribution here is expressed as MWIs. To define the latter, let the simplex ${\cal S}_t = \{(u_1, \ldots, u_r) \in \R^r: \, -\infty < u_1 < \ldots <u_r <t\}$ and $\{\B(u), \, u \in \R\}$ be a standard two-sided Brownian motion. For $1/ 2 < \beta < 1/2 + 1/({2r})$, define the {\it Hermite process} (Surgailis, 1982 and Avram and Taqqu, 1987) as the MWI
\begin{equation}\label{eqn:MWI}
Z_{r,\beta}(t) = \int_{{\cal S}_t} \int_0^t \prod_{i=1}^r g_\beta(v-u_i) d v\, d \B(u_1) \ldots d \B(u_r),
\end{equation}
where $g_\beta(x) = x^{-\beta}$ if $x > 0$ and $g_\beta(x) = 0$ if
$x \le 0$. It is non-Gaussian if $r \ge 2$. Note that $Z_{1,
\beta}(t)$ is the fractional Brownian motion with Hurst index $H =
3/2 - \beta$.

Let $\ell(n)$ be a slowly varying function, namely $\lim_{n \to
\infty} \ell(u n) /\ell(n) = 1$ for all $u > 0$ (Bingham, Goldie
and Teugels, 1987). Assume $a_0 \not= 0$ and $a_i$ has the form
\begin{equation}\label{eqn:ai}
a_i = i^{-\beta} \ell(i), \quad i\ge 1, \mbox{ where }1/2 < \beta < 1.
\end{equation}
Under (\ref{eqn:ai}), we say that $(a_i)$ is regularly varying with index $\beta$. Let $a_i = 0$ if $i < 0$, we need the following regularity condition on $K$ and the process
$(X_i)$.

\begin{condition}
\label{cond:2} For a function $f$ and $\lambda > 0$, write $f(x;
\lambda) = \sup_{|u| \le \lambda} |f(x+u)|$. Assume $\varepsilon_1
\in {\cal L}^{2 \nu}$ with $\nu \ge 2$, $K_{n} \in \mathcal C^{p+1}_\R$ for all large $n$, and for some $\lambda > 0$,
\begin{equation}\label{eqn:dn}
\sum_{\alpha=0}^{p+1} \| K_{n-1}^{(\alpha)}(X_{n,0}; \lambda) \|_\nu + \sum_{\alpha=0}^{p-1} \| \varepsilon_1^2 K_{n-1}^{(\alpha)} (X_{n,1}) \|_\nu + \| \varepsilon_1 K_{n-1}^{(p)} (X_{n,1}) \|_\nu = O(1).
\end{equation}
\end{condition}

We remark that in Condition \ref{cond:2} the function $K$ itself
does not have to be continuous. For example, if $K(x) = {\bf 1}_{x
\le 0}$; let $a_0 = 1$ and $F_\varepsilon$ (resp. $f_\varepsilon$)
be the distribution (resp. density) function of $\varepsilon_i$.
Then $K_1(x) = F_\varepsilon(-x)$ which is in ${\cal C}^{p+1}_\R$
if $F_\varepsilon$ is so. If $\sup_x |K_{n-1}^{(1+p)}(x)| <
\infty$, then for all $0 \le \alpha \le p$, there exists a
constant $C > 0$ such that $|K_{n-1}^{(\alpha)}(x)| \le C
(1+|x|)^{1+p-\alpha}$, and (\ref{eqn:dn}) holds if $\varepsilon_i
\in {\cal L}^{2\nu(1+p)}$.

\begin{theorem}{(Wu, 2006)}\label{thm:Wu06}
Assume that $K$ has power rank $p \geq 1$ with respect to $X_i$ and Condition \ref{cond:2} holds with
$\nu = 2$. (i) If $p (2\beta-1) < 1$, let
\begin{equation}\label{eqn:sgmanp}
\sigma_{n,p} = n^H \ell^p(n) \kappa_p \| Z_{p,\beta}(1) \|, \mbox{ where } H={1- p(\beta -1/2)},
\end{equation}
then in the space $\mathcal C_{[0,1]}$ we have the weak convergence
\begin{equation*}
\{S_{nt} / \sigma_{n,p}, ~ 0\le t\le 1\} \Rightarrow \{ Z_{p,\beta}(t) / \|Z_{p,\beta}(1)\|,\ 0 \le t\le 1\}.
\end{equation*}
(ii) If $p (2\beta -1) > 1$, then $D_0 := \sum_{j=0}^\infty {\cal
P}_0 Y_j \in {\cal L}^2. $ Assume $\|D_0\| > 0$. Then we have
\begin{equation}\label{eqn:lrdclt}
\{S_{nt} / \sigma_n,~ 0\le t\le 1\} \Rightarrow \{\B(t), ~ 0\le t\le 1\}, \mbox{ where } \sigma_n = \|D_0\| \sqrt n.
\end{equation}
\end{theorem}
The above result can not be directly applied for making statistical
inference for the mean $\mu = \E K(X_i)$ since $\sigma_{n,p}$ and
$\sigma_n$ are typically unknown. Additionally, the dichotomy in
Theorem \ref{thm:Wu06} causes considerable inconveniences in hypothesis
testings or constructing confidence intervals for $\mu$. The primary
goal of the paper is to establish the validity of some re-sampling techniques so that the distribution of $S_n$ can be
estimated.

\subsection{Block sampling}\label{subsec:blocksampling}
At the outset we assume that $\mu = \E K(X_i) = 0$. The block sampling method by Hall, Horowitz and Jing (1995)
can be described as follows. Let $l$ be the block size satisfying
$l = l_n \to \infty$ and $l / n \to 0$. For presentational
simplicity we assume that, besides $Y_1, \ldots, Y_n$, the past
observations $Y_{-l}, \ldots, Y_0$ are also available. Define
\begin{equation*}
s_l = \| S_l \|,
\end{equation*}
and the empirical distribution function
\begin{equation}\label{eqn:Fnx}
F_n(x) = {1\over n} \sum_{i=1}^n
 {\bf 1}_{Y_i + Y_{i-1} + \cdots + Y_{i-l+1} \le x s_l}.
\end{equation}
If $s_l$ is known, we say that the block sampling method is valid if
\begin{equation}\label{eqn:bsmvld}
\sup_{x \in \R} |F_n(x) - \p(S_n / s_n \le x)| \to 0 \mbox{ in probability.}
\end{equation}
In the long-memory case, the above convergence relation has a deeper layer of meaning since, by Theorem \ref{thm:Wu06}, $S_n/s_n$ can have either a Gaussian or non-Gaussian limiting distribution. In comparison, for short-memory processes, typically $S_n/s_n$ has a Gaussian limit. Ideally, we hope that (\ref{eqn:bsmvld}) holds for both cases in Theorem \ref{thm:Wu06}. Then we do not need to worry about the dichotomy of which limiting distribution to use. As a primary goal of the paper, we show that this is indeed the case.

In practice, both $\mu = \E K(X_i)$ and $s_l$ are not known. We
can simply estimate the former by $\bar Y_n = \sum_{i=1}^n Y_i /
n$ and the latter by
\begin{equation}\label{eqn:Q1}
\tilde s_l^2 = { {\tilde Q_{n, l}} \over n}, \mbox{ where }
 \tilde Q_{n, l} = \sum_{i=1}^n
 |Y_i + Y_{i-1} + \cdots + Y_{i-l+1}-l \bar Y_n |^2.
\end{equation}
The realized version of $F_n(x)$ in (\ref{eqn:Fnx}) now has the form
\begin{equation*}
\tilde F_n(x) = {1\over n} \sum_{i=1}^n
 {\bf 1}_{Y_i + Y_{i-1} + \cdots + Y_{i-l+1} -l \bar Y_n
 \le x \tilde s_l},
\end{equation*}
and correspondingly (\ref{eqn:bsmvld}) becomes
\begin{equation}\label{eqn:O12451}
\sup_{x \in \R} |\tilde F_n(x)
 - \p(S_n / \tilde s_n \le x)| \to 0
 \mbox{ in probability.}
\end{equation}
Later in Section \ref{subsec:Hhat} we will propose a consistent
estimate $\tilde s_n$ of $s_n$. In Section \ref{sec:cesd} we shall
show that (\ref{eqn:bsmvld}) holds for both cases in Theorem
\ref{thm:Wu06}. This entails (\ref{eqn:O12451}) if estimates $\tilde
s_l$ and $\tilde s_n$ satisfy $\tilde s_l / s_l \to 1$ and $\tilde
s_n / s_n \to 1$ in probability and $l (\bar Y_n - \mu) =
o_\p(s_l)$. With (\ref{eqn:O12451}), we can construct the two-sided
$(1-\alpha)$-th $(0 < \alpha < 1)$ and the upper one-sided $(1-\alpha)$-th
confidence intervals for $\mu$ as $[\bar Y_n - \tilde q_{1-\alpha/2}
\tilde s_n / n, \, \bar Y_n - \tilde q_{\alpha/2} \tilde s_n / n]$
and $[\bar Y_n - \tilde q_{1-\alpha} \tilde s_n / n, \,\infty)$
respectively, where $\tilde q_\alpha$ is the $\alpha$-th sample
quantile of $\tilde F_n(\cdot)$.

\subsection{Consistency of empirical sampling distributions}\label{sec:cesd}
Let $(\varepsilon'_j)_{j \in \Z}$ be an iid copy of
$(\varepsilon_j)_{j \in \Z}$, hence $\varepsilon'_i,
\varepsilon_l$, $i,l \in \Z$, are iid; let
\begin{equation}\label{eqn:Xistar}
X_i^* = X_i + \sum_{j=-\infty}^0 a_{i-j} (\varepsilon_j' - \varepsilon_j).
\end{equation}
Recall $a_j = 0$ if $j < 0$. We can view $X_i^*$ as a coupled
process of $X_i$ with $\varepsilon_j$, $j \leq 0$, in the latter
replaced by their iid copies $\varepsilon'_j$, $j \leq 0$. Note that, if $i
\le 0$, the two random variables $X_i$ and $X_i^* =
\sum_{j=0}^\infty a_j \varepsilon_{i-j}'$ are independent of each
other. Following Wu (2005), we define the physical dependence measure
\begin{equation}\label{eqn:taun}
\tau_{i, \nu} = \| K(X_i) - K(X_i^*) \|_\nu,
\end{equation}
which quantifies how the process $Y_i = K(X_i)$ forgets the past $\varepsilon_j$, $j \leq 0$.

\begin{theorem}\label{thm:L1}
Assume $\mu = \E Y_i = 0$, $p \geq 1$, $l \asymp n^{r_0}$ for some $0 < r_0 < 1$, and Condition \ref{cond:2} holds with
$\nu = 2$. (i) If $p (2\beta-1) < 1$, then
\begin{equation}\label{eqn:BSTPL}
\sup_{x \in \R} \left| F_n(x) - \p(Z_{p, \beta}(1) \le x ) \right| \to 0 \mbox{ in probability}.
\end{equation}
(ii) Let $Z \sim N(0, 1)$ be standard Gaussian. If $p (2\beta-1) >
1$, we have
\begin{equation*}
\sup_{x \in \R} \left| F_n(x) - \p(Z \le x ) \right| \to 0 \mbox{ in probability}.
\end{equation*}
Hence under either (i) or (ii), we have (\ref{eqn:bsmvld}).
\end{theorem}

As a useful and interesting fact, we emphasize from Theorem
\ref{thm:L1} that $F_n(\cdot)$ consistently estimates the
distribution of $S_n / s_n$, regardless of whether the limiting
distribution of the latter is Gaussian or not. In other words,
$F_n(\cdot)$ automatically adapts the limiting distribution of
$S_n / s_n$. Bertail, Politis and Romano (1999) obtained a result
of similar nature for strong mixing processes where the limiting
distribution can possibly be non-Gaussian; see also Politis,
Romano and Wolf (1999).

\bigskip

\begin{proof}
(Theorem \ref{thm:L1}) For (i), note that $Z_{p,\beta}(1)$ has a continuous distribution, by the Glivenko-Cantelli argument (cf. Chow and Teicher, 1997) for the uniform convergence of empirical distribution functions, (\ref{eqn:BSTPL}) follows if we can show that, for any fixed $x$,
\begin{equation*}
\E|F_n(x) - \p(Z_{p, \beta}(1) \le x )|^2 = {\rm var}(F_n(x)) + |\E F_n(x) - \p(Z_{p, \beta}(1) \le x )|^2 \to 0.
\end{equation*}
Let $B_{i, l} = Y_i + Y_{i-1} + \ldots + Y_{i-l+1}$. Since $B_{i, l}/ s_l \Rightarrow Z_{p, \beta}(1)$ as $n \to \infty$, the second term on the right hand side of the above converges to $0$. We now show that the first term
\begin{equation}\label{eqn:60001}
{\rm var} (F_n(x)) \leq {2\over n} \sum_{i=0}^{n-1} |{\rm cov}({\bf 1}_{B_{0, l} / s_l \le x},\ {\bf 1}_{B_{i, l} / s_l \le x}) | \to 0.
\end{equation}
Here we use the fact that $(B_{i,l})_{i \in \Z}$ is a stationary
process. To show (\ref{eqn:60001}), we shall apply the tool of
coupling. Recall (\ref{eqn:Xistar}) for $X_i^*$. Let $B_{i, l}^* =
\sum_{j=i-l+1}^i Y_j^*$, where $Y_j^* = K(X_j^*)$. Since $B_{i,
l}^*$ and ${\cal F}_{-\infty}^0$ are independent, $\E ({\bf
1}_{B^*_{i, l} / s_l \le x} | {\cal F}_{-\infty}^0)= \p(B^*_{i, l}
/ s_l \le x)$. Hence
\begin{eqnarray}\label{eqn:O2471}
 |{\rm cov}({\bf 1}_{B_{0, l} / s_l \le x},
 {\bf 1}_{B_{i, l} / s_l \le x}) |
 &=&  |\E[{\bf 1}_{B_{0, l} / s_l \le x}
 ({\bf 1}_{B_{i, l} / s_l \le x}
 - {\bf 1}_{B^*_{i, l} / s_l \le x} )] | \cr
 &\le&  \E |{\bf 1}_{B_{i, l} / s_l \le x}
 - {\bf 1}_{B^*_{i, l} / s_l \le x} |.
\end{eqnarray}
For any fixed $\lambda > 0$, by the triangle and the Markov
inequalities,
\begin{eqnarray}\label{eqn:O2472}
\E |{\bf 1}_{B_{i, l} / s_l \le x}
 - {\bf 1}_{B^*_{i, l} / s_l \le x} |
 &\le& \E ({\bf 1}_{|B_{i, l} / s_l - x| \le \lambda})
 + \E ({\bf 1}_{ |B_{i, l} / s_l - B^*_{i, l} / s_l|
  \ge \lambda} ) \cr
 &\le& \p(|B_{i, l} / s_l - x| \le \lambda)
 + { {\| B_{i, l} - B^*_{i, l}\|} \over {\lambda s_l}}.
\end{eqnarray}
Since $\E(B_{i, l} | {\cal F}_1^\infty) = \E(B^*_{i, l} | {\cal F}_1^\infty)$ for $i > 2 l$, by Lemma \ref{lem:S6511}(ii) and the fact that $B^*_{i, l} - \E(B^*_{i, l} | {\cal F}_1^\infty)$ and $B_{i, l}- \E(B_{i, l} | {\cal F}_1^\infty) $ are identically distributed, we have
\begin{eqnarray}\label{eqn:O3171}
\| B_{i, l}- B^*_{i, l} \| & \leq & \| B_{i, l}- \E(B_{i, l} | {\cal F}_1^\infty) \| + \| \E(B_{i, l} | {\cal F}_1^\infty) - B^*_{i, l}\| \cr
 & = & 2 \| B_{i, l}- \E(B_{i, l} | {\cal F}_1^\infty) \| \cr
 & = & 2 \|S_l - \E(S_l | {\cal F}_{l+1-i}^\infty) \| \cr
 & = & s_l O[l^{-\varphi_1} + (l/i)^{\varphi_2}].
\end{eqnarray}
Assume without loss of generality that $\varphi_2 < 1$. Otherwise we can replace it by $\varphi_2' = \min(\varphi_2, 1/2)$. By Lemma
\ref{lem:S6511}(i) and Lemma \ref{lem:snS}, we have $\| B_{0, l} \| = O(s_l)$. Recall that $l \asymp n^{r_0}$, $0 < r_0 < 1$, we have
\begin{eqnarray}\label{eqn:O3131}
{1\over n} \sum_{i=0}^{n-1}
 { {\| B_{i, l} - B^*_{i, l}\|} \over {s_l}}
 &=& {{O(1)} \over n }
  \sum_{i=0}^{2l} O(1)
 + {{O(1)} \over n } \sum_{i=2l+1}^{n-1}
  O[l^{-\varphi_1} + (l/i)^{\varphi_2}] \cr
  &=& O(l/n) + O(l^{-\varphi_1})
  + O[(l/n)^{\varphi_2}] = O(n^{-\phi}),
\end{eqnarray}
where $\phi = \min(1-r_0, \varphi_1 r_0, (1-r_0) \varphi_2)$.
Since $\p( |B_{i, l}/ s_l - x| \le \lambda) \to \p( |Z_{p,
\beta}(1) - x| \le \lambda)$, (\ref{eqn:60001}) then follows from
(\ref{eqn:O2471}) and (\ref{eqn:O2472}) by first letting $n \to
\infty$, and then $\lambda \to 0$.

For (ii), by the argument in (i), it suffices to show that
\begin{equation}\label{eqn:O3161}
\lim_{n \to \infty} {1\over n} \sum_{i=1}^n
 { {\| B_{i, l} - \E(B_{i, l} | {\cal F}_1^\infty)
 \|} \over { \sqrt l}} = 0.
\end{equation}
More specifically, if (\ref{eqn:O3161}) is valid, then by $\|
B_{i, l}- B^*_{i, l} \| \le 2 \| B_{i, l}- \E(B_{i, l} | {\cal
F}_1^\infty) \|$, we have (\ref{eqn:O3131}) and consequently
(\ref{eqn:60001}).

Let $N > 3 l$ and $G_N = B_{N, l}-\E(B_{N,l}|{\cal F}_1^\infty)$.
Observe that $({\cal P}_k G_N)_{k=-\infty}^N$ is a sequence of
martingale differences and $G_N = \sum_{k=-\infty}^N {\cal P}_k G_N$,
we have
\begin{equation}\label{eqn:O3331}
\| G_N \|^2 = \sum_{k=-\infty}^N \| {\cal P}_k G_N \|^2.
\end{equation}
By (\ref{eqn:54401}) and Lemma \ref{lem:prdmcexp} with $\nu = 2$, we
know that the predictive dependence measures $\eta_i = \|{\cal
P}_0 Y_i\|$ is summable. Recall (\ref{eqn:taun}) for $\tau_{n,
\nu}$. Let $\tau^*_n = \max_{m \ge n} \tau_{m, 2}$. Then
$\tau^*_n$ is non-increasing and $\lim_{n \to \infty} \tau^*_n =
0$. Since $\|{\cal P}_k \E(Y_j | {\cal F}_1^\infty) \| \le \|{\cal
P}_k Y_j \| = \eta_{j-k}$ and $\|Y_j - \E(Y_j | {\cal F}_1^\infty)
\| \le \tau_{j, 2}$, we have
\begin{eqnarray}\label{eqn:O3341}
\|{\cal P}_k G_N\| &\le&
 \sum_{j=N-l+1}^N \|{\cal P}_k [Y_j - \E(Y_j | {\cal F}_1^\infty) ] \| \cr
  &\le& \sum_{j=N-l+1}^N \min(2 \eta_{j-k}, \, \tau^*_{N-l+1}) \le \eta_*,
\end{eqnarray}
where $\eta_* = 2 \sum_{i=0}^\infty \eta_i$. Then, by (\ref{eqn:O3331}) and the Lebesgue dominated convergence theorem, we have
\begin{eqnarray}\label{eqn:O3471}
\lim_{n \to \infty} { {\| G_N \|^2} \over l}
 & \leq & \lim_{n \to \infty}
  \sum_{k=-\infty}^N {\eta_* \over l} \| {\cal P}_k G_N \| \cr
 &\leq& \lim_{n \to \infty} \sum_{k=-\infty}^N {\eta_*\over l}
      \sum_{j=N-l+1}^N \min(2 \eta_{j-k}, \, \tau^*_{N-l+1}) \cr
 & \leq & \lim_{n \to \infty} \eta_* \sum_{i=0}^\infty \min(2 \eta_i, \, \tau^*_{N-l+1}) = 0,
\end{eqnarray}
since $\tau^*_{N-l+1} \le \tau^*_{l} \to 0$ as $l \to \infty$ and $\eta_i$ are summable. Hence $\sum_{N=3 l}^n \| G_N \|^2 = o(n l)$. Note that $l = o(n)$, (\ref{eqn:O3161}) follows by the inequality $(\sum_{i=1}^n |z_i| / n)^2 \leq \sum_{i=1}^n z_i^2 / n$.
\end{proof}

\subsection{Variance estimation}\label{subsec:sltilde}
Since $F_n(\cdot)$ and the relation (\ref{eqn:bsmvld}) involve unknown quantities $s_l$ and $s_n$, Theorem \ref{thm:L1} is not directly applicable for making statistical inferences on $\mu$, while it implies (\ref{eqn:O12451}) if we can find estimates $\tilde s_l$ and $\tilde s_n$ such that $\tilde s_l/s_l \to 1$ and $\tilde s_n / s_n \to 1$ in probability and $l(\bar Y_n - \mu) = o_\p(s_l)$. We propose to estimate $s_l$ by using (\ref{eqn:Q1}); see Theorem \ref{eqn:Sgm2} for the asymptotic properties of the variance estimate $\tilde s_l^2$. However, there is no analogous way to propose a consistent estimate for $s_n$ since one can not use blocks of size $n$ to estimate it. One way out is to use its regularly varying property (cf. equations (\ref{eqn:O38231}) and (\ref{eqn:O38251})) via estimating the self-similar parameter $H$ (see Section \ref{subsec:Hhat}). Section \ref{sec:S19950} proposes a subsampling approach which does not require estimating $H$. Recall (\ref{eqn:sgmanp}) and (\ref{eqn:lrdclt}) for the definitions of $\sigma_{n, p}$ and $\sigma_n$, respectively. Lemma \ref{lem:snS} asserts that they are asymptotically equivalent to $s_n$.

\begin{lemma}\label{lem:snS}
Recall that $s_l = \| S_l \|$. Under conditions in Theorem \ref{thm:Wu06}(i), we have
\begin{equation}\label{eqn:O38231}
s_l \sim \sigma_{l, p} = l^H \ell^p(l)  \kappa_p \| Z_{p,\beta}(1) \|,
\end{equation}
as $l \to \infty$. Under conditions in Theorem \ref{thm:Wu06}(ii), we have
\begin{equation}\label{eqn:O38251}
s_l \sim \sigma_{l} = \|D_0\| \sqrt l .
\end{equation}
Under either case, $l \|\bar Y_n - \mu \| = o(s_l)$ if $l \asymp n^{r_0}$, $0 < r_0 < 1$.
\end{lemma}

If $\mu = \E Y_i$ is known, say $\mu = 0$, then we can estimate
$s^2_l$ by
\begin{equation*}
\hat s_l^2 = { {\hat Q_{n, l}} \over n}, \mbox{ where } \hat Q_{n, l} = \sum_{i=1}^n |Y_i + Y_{i-1} + \ldots + Y_{i-l+1}|^2.
\end{equation*}
Clearly $\hat s_l^2$ is an unbiased estimate of $s^2_l =
\|S_l\|^2$. Theorem \ref{eqn:Sgm2} provides a convergence rate of
the estimate. As a simple consequence, we know that $\hat s_l^2$
is consistent.

\begin{theorem}\label{eqn:Sgm2}
Assume that $l \asymp n^{r_0}$, $0 < r_0 < 1$, and Condition \ref{cond:2} holds with $\nu = 4$. (i) If $p(2 \beta - 1) < 1$, then there exists a constant $0 < \phi < 1$ such that
\begin{equation}\label{eqn:O24801}
{\rm var} (\tilde s_l^2 / s_l^2) = O(n^{-\phi}).
\end{equation}
(ii) If $p(2 \beta - 1) > 1$, then ${\rm var} (\tilde s_l^2/ s_l^2) \to 0$. (iii) If $p (2\beta-1) > 1$ and $\tau_{n,4} = O(n^{-\phi_1})$ for some $\phi_1 > 0$, then (\ref{eqn:O24801}) holds as well.
\end{theorem}

\begin{proof}
(Theorem \ref{eqn:Sgm2}) For (i), we first consider the case with $\mu = 0$ and show that, for some $\phi > 0$,
\begin{equation}\label{eqn:S11501}
{\rm var} (\hat s_l^2 / s_l^2) = O(n^{-\phi}).
\end{equation}
Recall that $B_{n, l}^* = \sum_{j=n-l+1}^n Y_j^*$, where $Y_j^* =
K(X_j^*)$. Then $\E (B^2_{i, l}) = \E [(B^*_{i, l})^2 | {\cal
F}_0]$ and ${\rm cov}( B^2_{0, l}, B^2_{i, l}) = \E[B^2_{0, l}
(B^2_{i, l}- (B^*_{i, l})^2)]$. By the Cauchy-Schwarz inequality,
\begin{eqnarray}\label{eqn:S20501}
{\rm var} (\hat s_l^2)
 & = & {1 \over n} \sum_{i=1-n}^{n-1}
    (1- |i|/n) {\rm cov}( B^2_{0, l}, \, B^2_{i, l}) \cr
 &\leq& {2 \over n} \sum_{i=0}^{n-1}
   \|B^2_{0,l}\| \| B^2_{i, l}- (B^*_{i, l})^2) \| \cr
 & \leq & {2 \over n} \sum_{i=0}^{n-1} \| B_{0, l} \|_4^2
 \| B_{i, l} + B^*_{i, l} \|_4 \| B_{i, l}- B^*_{i, l} \|_4.
\end{eqnarray}
By Lemma \ref{lem:S6511}(ii) and the argument (\ref{eqn:O3171}) in
the proof of Theorem \ref{thm:L1}(i), for $i > 2 l$, we have
\begin{eqnarray}\label{eqn:O097301}
\| B_{i, l}- B^*_{i, l} \|_4
 &\leq& \| B_{i, l}- \E(B_{i, l} | {\cal F}_0^\infty) \|_4
 + \| \E(B_{i, l} | {\cal F}_0^\infty) - B^*_{i, l}\|_4 \cr
 &=& s_l O[l^{-\varphi_1} + (l/i)^{\varphi_2} ],
\end{eqnarray}
in view of Lemma \ref{lem:snS} since $\|B_{i, l} \| \sim s_l$. Again we assume without loss generality that $\varphi_2 < 1$. By Lemma \ref{lem:S6511}(i),
$\| B_{0, l} \|_4 = O(\sigma_{l, r})$. So (\ref{eqn:S20501})
similarly implies (\ref{eqn:S11501}) via
\begin{eqnarray}\label{eqn:O097331}
{\rm var} (\hat s_l^2 / s_l^2)
 &=& {{O(1)} \over n }
  \sum_{i=0}^{2l} O(1)
 + {{O(1)} \over n } \sum_{i=2l+1}^{n-1}
  O[l^{-\varphi_1} + (l/i)^{\varphi_2}] \cr
  &=& O(l/n) + O(l^{-\varphi_1})\cr
  + O ( (l/n)^{\varphi_2} )
  &=& O(n^{-\phi})
\end{eqnarray}
with $\phi = \min(1-r_0, \, \varphi_1 r_0, \, (1-r_0) \varphi_2)$
since $l \asymp n^{r_0}$, $0 < r_0 < 1$.

Now we shall show that (\ref{eqn:S11501}) implies
(\ref{eqn:O24801}). By Lemma \ref{lem:S6511}(i) and the
Cauchy-Schwarz inequality,
\begin{eqnarray}\label{eqn:O097591}
\| \hat Q_{n, l} - \tilde Q_{n, l} \|
 & = & \left\|n (l \bar Y_n)^2 -2 l  \bar Y_n \sum_{i=1}^n B_{i, l}\right\| \cr
 &\leq& n l^2 \|\bar Y_n^2\| + \|2l \bar Y_n \|_4 l \| Y_1 + \cdots + Y_n \|_4 \cr
 & = & O(l^2 s_n^2 / n) \cr
 &= & n s_l^2  O( l^2 s_n^2 / (n^2 s_l^2)) \cr
 & = & n s_l^2 O [(l/n)^{2-2 H} \ell^{2p}(n) / \ell^{2p}(l)] \cr
 & = & n s_l^2 O (n^{-\theta}),
\end{eqnarray}
where $0 < \theta < (2-2H)(1-r_0)$. Hence (\ref{eqn:O24801}) follows from Lemma \ref{lem:snS}.

For (iii), by (\ref{eqn:prdmV}) and (\ref{eqn:54401}), under $p(2\beta - 1) > 1$, for $0 < \varphi_3 < p (2\beta-1)$, the predictive dependence measure
\begin{eqnarray*}
\eta_{i, 4} &:=& \| {\cal P}_0 Y_i \|_4 = \|{\cal P}_0 (L_{n, p} + \kappa_p U_{n, p}) \|_4 \leq |\kappa_p| \|{\cal P}_0 U_{n, p} \|_4 + \|{\cal P}_0 L_{n, p} \|_4 \cr
 & = & O(a_n A_n^{(p-1)/2}) + a_n O(a_n + A^{1/2}_{n+1}(4) + A_{n+1}^{p/2}) \cr
 & = & O(i^{-1-\varphi_3}),
\end{eqnarray*}
where $L_{n,p}$ is defined in (\ref{eqn:lq110}). Recall the proof of Theorem \ref{thm:L1}(ii) for the definition of $G_N$, $N > 3 l$. By (\ref{eqn:DMIM}), $\| G_N \|^2_4 \le C_4 \sum_{k=-\infty}^N \| {\cal P}_k G_N \|_4^2$, and the arguments in (\ref{eqn:O3341}) and (\ref{eqn:O3471}), there exists a constant $C > 0$ such that
\begin{equation*}
{\| G_N \|_4^2 \over l} \le C \eta_{*, 4} \sum_{i=0}^\infty \min(\eta_{i, 4}, \, \tau^*_{N-l+1, 4}),
\end{equation*}
where $\eta_{*,4} = \sum_{i=0}^\infty \eta_{i, 4}$ and $\tau^*_{n,
4} = \max_{m \ge n} \tau_{m, 4}$. As $\tau \to 0$, we have
$\sum_{i=0}^\infty \min(\eta_{i, 4},\,\tau) =O(\tau^{\varphi_4})$,
where $\varphi_4 = \varphi_3 / (1+ \varphi_3)$. Similarly as
(\ref{eqn:S20501}), (\ref{eqn:O097301}) and (\ref{eqn:O097331}),
\begin{eqnarray*}
{\rm var} (\hat s_l^2 / s_l^2) & = &
 {{O(1)} \over n}  \sum_{i=0}^{n-1} {\| G_i \|_4 \over \sqrt l} \cr
 &=& {O(1) \over n} \sum_{i=1+3 l}^{n-1} {\| G_i \|_4 \over \sqrt l} + O(l/n) \cr
 & = & {{O(1)} \over n} \sum_{i=1+3 l}^{n-1} i^{-\phi_1 \varphi_4/2} + O(l/n) \cr
 & = & O(n^{-\phi_1 \varphi_4/2}) + O(l/n).
\end{eqnarray*}
So (\ref{eqn:S11501}), and hence (\ref{eqn:O24801}) follows in view of (\ref{eqn:O097591}).

For (ii), as in the proof Theorem \ref{thm:L1}(ii), it follows from the Lebesgue dominated convergence theorem since $\tau^*_{m, 4} \to 0$ as $m \to \infty$.
\end{proof}

\subsection{Estimation of $H$}\label{subsec:Hhat}
In the study of self-similar or long-memory
processes, a fundamental problem is to estimate $H$, the
self-similar parameter. The latter problem has been extensively
studied in the literature. The approach of spectral estimation
which uses periodograms to estimate $H$ has been considered, for
example, by Robinson (1994, 1995a and 1995b) and Moulines and Soulier
(1999). To extend the case where the underlying process is or
close to linear, Hurvich, Moulines and Soulier (2005) deals with a
nonlinear model widely used in econometrics which contains a
latent long-memory volatility component. Taking the time-domain
approach, Teverovsky and Taqqu (1997) and Giraitis, Robinson and
Surgalis (1999) focus on a variance--type estimator for $H$. Here
we shall estimate $H$ based on $\sigma_{l, p}$ by using a two-time-scale method to estimate $H$. By Lemma \ref{lem:snS},
\begin{equation*}
\lim_{l \to \infty} {s_{2l} \over s_{l}} = \lim_{l \to \infty} {\sigma_{2l, p} \over \sigma_{l, p}} = 2^H.
\end{equation*}
Based on Theorem \ref{eqn:Sgm2}, we can estimate $H$ by
\begin{equation*}
\hat H = { {\log \hat s_{2l} - \log \hat s_{l} } \over \log 2}.
\end{equation*}
Corollary \ref{cor:H} asserts that $\hat H$ is a consistent
estimate of $H$. To obtain a convergence rate, we need to impose
regularity conditions on the slowly varying function $\ell(
\cdot)$. The estimation of slowly varying functions is a highly
non-trivial problem. In estimating $\sigma_n$ in the context of
linear processes or nonlinear functionals of Gaussian processes, Hall, Jing and Lahiri (1998) and Nordman and Lahiri (2005) imposed some
conditions on $\ell$. In our setting, for the sake of readability,
we assume that $\ell(n) \to c_0$, though our argument can be
generalized to deal with other $\ell$ with some tedious
calculations. Under Condition \ref{cond:1}, by Lemma
\ref{lem:S096311}(iii), $\sigma_{l, p} / (l^H c_0) = 1 + O(
l^{-\varphi_2})$. So we estimate $s_n$ by
\begin{equation*}
\hat \sigma_{n, p} = n^{\hat H} \hat c_0, \mbox{ where }
  \hat c_0 = { {\hat \sigma_{l, p} } \over {l^{\hat H}}}.
\end{equation*}
In practice we can choose $l = \lfloor c n^{1/2} \rfloor$ for some $0 < c < \infty$. The problem of choosing an optimal data-driven $l$ is beyond the scope of the current paper.

\begin{condition}\label{cond:1}
The coefficients $a_0 \not= 0$, $a_j = c_j j^{-\beta}$, $j \geq 1$, where $1/2 < \beta < 1$ and $c_j = c + O(j^{-\phi})$ for some $\phi > 0$.
\end{condition}

Condition \ref{cond:1} is satisfied by the popular FARIMA processes.

\begin{corol}
\label{cor:H} Assume that $l \asymp n^{r_0}$, $0 < r_0 < 1$, and
Condition \ref{cond:2} holds with $\nu = 4$. (i) Under either $p (
2 \beta - 1) < 1$ or $p ( 2 \beta - 1) > 1$, we have $\lim_{n \to
\infty} \hat H = H$. (ii) Assume $p ( 2 \beta - 1) < 1$ and
Condition \ref{cond:1}. Then
\begin{eqnarray}\label{eqn:S9441}
\hat H - H = O(n^{-\phi})
\end{eqnarray}
and
\begin{eqnarray}\label{eqn:S96111}
{ {\hat s_{n}} \over {s_{n}}} \to 1 \mbox{ in probability.}
\end{eqnarray}
(iii) Under conditions of Theorem \ref{eqn:Sgm2}(iii), we have (\ref{eqn:S9441}) with $H = 1/2$ and (\ref{eqn:S96111}).
\end{corol}

\begin{proof}
For (i), by Theorem \ref{eqn:Sgm2}(i, ii) and Lemma
\ref{lem:snS}, we have $\E| \tilde s_l^2 / \varpi_l^2 - 1|^2
\to 0$, where $\varpi_l = \sigma_{l, p}$ if $p ( 2 \beta - 1) < 1$ and $\varpi_l = \sigma_{l}$ if $p ( 2 \beta - 1) > 1$. Hence
$\tilde s_l^2 / \varpi_l^2 = 1 + o_\p(1)$ and $\tilde s_{2 l}^2 /
\tilde s_l^2 = \varpi_{2 l}^2/ \varpi_l^2 + o_\p(1) = 2^{2 H} +
o_\p(1)$. Thus $\lim_{n \to \infty} \hat H = H$.

For (ii), under Condition \ref{cond:1}, we have $s_l^2 / \sigma_{l,p}^2
= 1 + O(n^{-\phi})$, which by Theorem \ref{eqn:Sgm2}(i)
implies that $\tilde s_l / \sigma_{l,p} = 1 + O_\p(n^{-\phi})$ and
hence $\tilde s_{2l}^2/\tilde s_l^2 = 2^{2 H} + O_\p(n^{-\phi})$.
So (\ref{eqn:S9441}) follows. For (\ref{eqn:S96111}), by
(\ref{eqn:S9441}), we have $l^{\hat H} / l^H = 1 + O(n^{-\phi r_0}
\log n)$. Hence for some $\phi_4 > 0$, we have $\hat c_0 / c_0 = 1
+ O(n^{-\phi_4})$, which entails (\ref{eqn:S96111}) in view of
$n^{\hat H} / n^H = 1 + O(n^{-\phi_1} \log n)$.

For (iii), let $D_k = \sum_{i=k}^\infty {\cal P}_k Y_i$. Recall from
the proof of Theorem \ref{eqn:Sgm2}(iii) that $\eta_{i, 4} = \|
{\cal P}_0 Y_i \|_4 = O(i^{-1-\varphi_3})$. By Theorem 1 in Wu
(2007), $\|S_l - \sum_{i=1}^l D_i\|^2_4 = \sum_{i=1}^l
O(\Theta^2_i)$, where $\Theta_i = \sum_{j=i}^\infty \eta_{i,4} =
O(i^{-\varphi_3})$. Hence $\|S_l - \sum_{i=1}^l D_i\| / \sqrt l =
O(l^{-\varphi_3})$, which implies that $s_l / \sigma_l -1 =
O(l^{-\varphi_3})$. Then the result follows from the argument in
(ii) and Theorem \ref{eqn:Sgm2}(iii).
\end{proof}

\section{A Subsampling Approach}\label{sec:S19950}
The block sampling method in Section \ref{subsec:blocksampling} requires consistent estimation of $s_l$ and $s_n$. The former is treated in Section \ref{subsec:sltilde}, while the latter is achieved by estimating the self-similar parameter $H$; see Section \ref{subsec:Hhat}. Here we shall propose a subsampling method which can directly estimate the distribution of $S_n$ without having to estimate $H$. To this end, we choose positive integers $n_1$ and $l_1$ such that
\begin{equation}\label{eqn:S91022}
{l_1 \over n_1} = {l \over n}, \mbox{ and } {1 \over l_1} + {n_1+l \over n} = O(n^{-\theta}) \mbox{ for some } \theta > 0.
\end{equation}
Further assume that $\ell(\cdot)$ is {\it strongly slowly varying} in the sense that $\lim_{k \to \infty} \ell(k) / \ell(k^\alpha) =
1$ for any $\alpha > 0$. It holds for functions like $\ell(k) = (\log \log k)^c$, $c \in \R$, while the slowly varying function $\ell(k) = \log k$ is not strongly slowly varying. Similar conditions were also used in Hall, Jing and Lahiri (1998) and Nordman and Lahiri (2005). Note that (\ref{eqn:S91022}) implies that
\begin{equation}\label{eqn:S10838}
\lim_{n \to \infty} {s_{l_1} s_n \over s_{l} s_{n_1}} = 1.
\end{equation}
Then by Theorem \ref{thm:Wu06} and condition (\ref{eqn:S91022}), we have
\begin{eqnarray}\label{eqn:subsampling}
& & \sup_{u \in \mathbb R} |\mathbb P(S_n/s_{n_1} \leq u)
 - \mathbb P(S_l/s_{l_1} \le u)| \cr
& & = \sup_{x \in \mathbb R} |\mathbb P((S_n/s_{n_1}) (s_{n_1}/s_n) \leq x)
 - \mathbb P((S_l/s_{l_1}) (s_{n_1}/s_n) \leq x)| \to 0.
\end{eqnarray}
Hence, the distribution of $S_n/s_{n_1}$ can be approximated by that of $S_l/s_{l_1}$. Let
\begin{equation}
\tilde F_l^\star(x) = {1 \over n} \sum_{i=1}^{n-l+1} \mathbf{1}_{\{(\sum_{j=i}^{i+l-1} Y_j - l\bar Y_n)/\tilde s_{l_1,i} \leq x\}},
\end{equation}
where
\begin{equation}\label{eqn:sl1itilde}
\tilde s_{l_1,i}^2 = { {\tilde Q_{l,l_1,i}} \over l-l_1+1},
 \mbox{ with } \tilde Q_{l,l_1,i} = \sum_{j=1}^{l-l_1+1}
  |Y_{i+j-1} + \cdots + Y_{i+j+l_1-2} - l_1 \bar Y_n|^2.
\end{equation}
Since $\lim_{n \to \infty} \tilde s_{l_1,i}^2 / s^2_{l_1} = 1$, using the argument in Theorem \ref{thm:L1}, we have
\begin{equation}
\sup_x |\tilde F_l^\star(x) - \mathbb P(S_l/s_{l_1} \leq x) |
 \to 0 \mbox{ in probability.}
\end{equation}
Note that  $s_{n_1}$ can be estimated by (\ref{eqn:Q1}). Then confidence intervals for $\mu$ can be constructed based on sample quantiles of $\tilde F_l^\star(\cdot)$.

\section{Simulation Study}\label{sec:simu}
Consider a stationary process $Y_i = K(X_i)$, where $X_i$ is a linear process defined in (\ref{eqn:Xidef}) with $a_k = (1+k)^{-\beta}$, $k \ge 0$, and $\varepsilon_i$, $i \in \mathbb Z$, are iid innovations. We shall here investigate the finite-sample performance of the block sampling method described in Section \ref{sec:mainsec} (based on $\hat H$) and \ref{sec:S19950} (based on subsampling) by considering different choices of the transform $K(\cdot)$, the beta index $\beta$, the sample size $n$ and innovation distributions. In particular, we consider the following four processes:
\begin{itemize}
\item[(i)] $K(x) = x$, and $\epsilon_i$, $i \in \mathbb Z$, are iid $N(0,1)$;
\item[(ii)] $K(x) = \boldsymbol 1_{\{x \leq 1\}}$, and $\epsilon_i$, $i \in \mathbb Z$, are iid $t_7$;
\item[(iii)] $K(x) = \boldsymbol 1_{\{x \leq 0\}}$, and $\epsilon_i$, $i \in \mathbb Z$, are iid $t_7$;
\item[(iv)] $K(x) = x^2$, and $\epsilon_i$, $i \in \mathbb Z$, are iid Rademacher.
\end{itemize}
For cases (i) and (ii), the power rank $p = 1$, while for (iii) and (iv), the power rank $p = 2$. If $p = 1$, we let $\beta = 0.75$ and $\beta = 2$, which correspond to long- and short-range dependent processes, respectively. For $p = 2$, we consider three cases: $\beta \in \{0.6,0.8,2\}$. The first two are situations of long-range dependence but have different limiting distributions as indicated in Theorems \ref{thm:Wu06} and \ref{thm:L1}. We use block sizes $l = \lfloor c n^{0.5} \rfloor$, $c \in \{0.5,1,2\}$, and $n_1 = \lfloor n^{0.9} \rfloor$. Let $n \in \{100, 500, 1000\}$. The empirical coverage probabilities of lower and upper one-sided 90\% confidence intervals are computed based on $5,000$ realizations and they are summarized in Table \ref{tab:Simulation} as pairs in parentheses. We observe the following phenomena. First, the accuracy of the coverage probabilities generally improves as we increase $n$, or decrease the strength of dependence (increasing the beta index $\beta$). Second, the nonlinearity worsens the accuracy, noting that the processes in (ii)--(iv) are nonlinear while the one in  (i) is linear. Lastly, the subsampling-based procedure described in Section \ref{sec:S19950} usually has a better accuracy than the one based on $\hat H$ as described in Section \ref{sec:mainsec}.

\begin{table}\footnotesize
\centering
\begin{tabular}{ccccccccc}
\hline \hline
 & & \multicolumn{3}{c}{$\hat H$-based} & & \multicolumn{3}{c}{Subsampling-based} \\ \cline{3-5} \cline{7-9}
$\beta$ & $n$ & $c = 0.5$ & $c = 1$ & $c = 2$ & & $c = 0.5$ & $c = 1$ & $c = 2$ \\
\hline\vspace{-0.5em}
 & & \multicolumn{7}{c}{Model (i)} \\\vspace{-0.5em}
0.75 & \phantom{0}100 & (79.5, 80.4) & (75.2, 74.5) & (69.6, 72.2) & & (90.8, 91.8) & (87.9, 87.2) & (83.5, 84.9) \\\vspace{-0.5em}
     & \phantom{0}500 & (85.0, 85.8) & (80.8, 81.3) & (75.9, 78.7) & & (93.5, 94.2) & (91.1, 91.6) & (88.0, 89.7) \\
     & 1000 & (86.7, 87.3) & (83.5, 82.2) & (80.3, 78.2) & & (94.5, 94.2) & (92.6, 91.9) & (90.5, 89.3) \\\vspace{-0.5em}
2\phantom{.00} & \phantom{0}100 & (89.1, 89.0) & (87.4, 87.1) & (84.1, 83.2) & & (93.6, 92.5) & (89.4, 89.3) & (85.6, 85.3) \\\vspace{-0.5em}
     & \phantom{0}500 & (90.9, 90.7) & (90.3, 90.4) & (89.2, 89.1) & & (93.6, 93.0) & (91.6, 92.3) & (89.8, 90.1) \\
     & 1000 & (91.7, 91.5) & (90.1, 91.4) & (90.2, 90.7) & & (93.1, 93.4) & (91.5, 92.6) & (90.6, 90.3) \\\vspace{-0.5em}

 & & \multicolumn{7}{c}{Model (ii)} \\\vspace{-0.5em}
0.75 & \phantom{0}100 & (60.9, 85.7) & (60.5, 82.6) & (59.2, 80.7) & & (100.0, 96.8)\phantom{1} & (99.9, 94.9) & (95.8, 90.5) \\\vspace{-0.5em}
     & \phantom{0}500 & (68.8, 90.1) & (69.1, 86.2) & (68.6, 83.2) & & (100.0, 98.3)\phantom{1} & (99.8, 95.7) & (97.0, 92.1) \\
     & 1000 & (71.5, 91.9) & (72.2, 89.3) & (71.7, 84.2) & & (100.0, 97.8)\phantom{1} & (99.8, 95.9) & (96.9, 92.7) \\\vspace{-0.5em}
2\phantom{.00} & \phantom{0}100 & (75.8, 93.2) & (81.3, 89.6) & (78.9, 87.7) & & (100.0, 89.9)\phantom{1} & (99.1, 86.3) & (89.5, 84.5) \\\vspace{-0.5em}
     & \phantom{0}500 & (88.5, 92.8) & (87.5, 92.4) & (86.9, 91.1) & & (100.0, 87.3)\phantom{1} & (95.1, 88.8) & (91.8, 87.3) \\
     & 1000 & (89.1, 93.0) & (88.0, 92.1) & (88.0, 91.5) & & (98.4, 88.0) & (95.1, 88.2) & (92.0, 87.7) \\\vspace{-0.5em}

 & & \multicolumn{7}{c}{Model (iii)} \\\vspace{-0.5em}
0.6 & \phantom{0}100 & (71.2, 70.9) & (69.6, 68.6) & (67.4, 68.4) & & (99.3, 98.2) & (99.2, 98.0) & (95.6, 94.7) \\\vspace{-0.5em}
     & \phantom{0}500 & (75.9, 75.3) & (74.6, 73.4) & (70.8, 72.2) & & (100.0, 99.9)\phantom{1} & (99.7, 99.7) & (97.3, 97.8) \\
     & 1000 & (78.0, 78.3) & (76.7, 76.3) & (75.3, 72.8) & & (100.0,100.0)\phantom{1} & (99.7, 99.7) & (98.3, 97.3) \\\vspace{-0.5em}
0.8 & \phantom{0}100 & (76.2, 74.4) & (73.7, 74.3) & (71.7, 71.7) & & (100.0,100.0)\phantom{1} & (97.9, 97.6) & (91.5, 90.6) \\\vspace{-0.5em}
     & \phantom{0}500 & (82.4, 83.6) & (81.3, 79.5) & (75.7, 77.8) & & (99.7, 99.7) & (98.0, 97.3) & (93.2, 94.0) \\
     & 1000 & (85.2, 84.7) & (83.0, 82.5) & (81.9, 77.8) & & (99.4, 99.5) & (97.5, 96.9) & (95.1, 93.4) \\\vspace{-0.5em}
2\phantom{.00} & \phantom{0}100 & (88.7, 89.0) & (87.8, 86.9) & (83.9, 84.1) & & (94.8, 94.7) & (91.8, 89.8) & (86.3, 86.6) \\\vspace{-0.5em}
     & \phantom{0}500 & (90.2, 90.7) & (90.3, 89.1) & (88.9, 89.3) & & (93.8, 93.6) & (92.1, 91.7) & (89.3, 89.5) \\
     & 1000 & (91.0, 91.5) & (90.1, 90.7) & (90.2, 89.5) & & (93.3, 93.9) & (91.5, 91.3) & (90.6, 89.4) \\\vspace{-0.5em}

 & & \multicolumn{7}{c}{Model (iv)} \\\vspace{-0.5em}
0.6 & \phantom{0}100 & (98.4, 34.0) & (97.3, 30.0) & (96.1, 31.6) & & (99.2, 77.9) & (98.5, 71.4) & (97.9, 65.4) \\\vspace{-0.5em}
     & \phantom{0}500 & (99.0, 42.6) & (97.7, 42.5) & (97.3, 36.4) & & (99.1, 88.6) & (98.4, 86.2) & (98.3, 77.0) \\
     & 1000 & (99.1, 48.2) & (98.2, 44.4) & (97.3, 42.5) & & (99.1, 91.9) & (98.9, 86.9) & (98.1, 82.1) \\\vspace{-0.5em}
0.8 & \phantom{0}100 & (96.5, 56.7) & (94.7, 56.8) & (92.9, 55.5) & & (97.4, 88.1) & (95.6, 85.0) & (93.3, 78.5) \\\vspace{-0.5em}
     & \phantom{0}500 & (97.9, 67.7) & (96.5, 65.8) & (94.4, 66.1) & & (96.6, 95.2) & (95.6, 93.4) & (93.1, 88.3) \\
     & 1000 & (98.2, 72.3) & (96.6, 68.3) & (95.0, 70.0) & & (95.6, 96.8) & (94.6, 93.9) & (93.2, 91.5) \\\vspace{-0.5em}
2\phantom{.00} & \phantom{0}100 & (94.8, 82.1) & (92.7, 81.9) & (88.4, 80.5) & & (86.6, 90.7) & (86.7, 89.1) & (84.0, 86.0) \\\vspace{-0.5em}
     & \phantom{0}500 & (94.2, 86.5) & (93.5, 87.5) & (91.9, 85.6) & & (85.7, 94.7) & (86.7, 93.6) & (86.1, 90.9) \\
     & 1000 & (94.2, 87.0) & (93.5, 87.9) & (91.7, 87.7) & & (87.3, 94.8) & (86.5, 93.6) & (86.2, 90.9) \\
\hline
\end{tabular}
\caption{Empirical coverage probabilities of lower and upper (paired in parentheses) one-sided 90\% confidence intervals for processes (i)--(iv) with different combinations of beta index $\beta$, sample size $n$ and block size $l = \lfloor c n^{0.5} \rfloor$.}\label{tab:Simulation}
\end{table}

\section{Appendix}\label{sec:proofs}
Recall that ${\cal F}_i^j = (\varepsilon_i,\varepsilon_{i+1}, \ldots, \varepsilon_j)$, $i \leq j$, ${\cal F}_i^\infty = (\varepsilon_i,\varepsilon_{i+1}, \ldots)$ and ${\cal F}_{-\infty}^j= (\ldots,\varepsilon_{j-1}, \varepsilon_j)$. In dealing with nonlinear functionals of linear processes, we will use the powerful tool of Volterra expansion (Ho and Hsing, 1997 and Wu, 2006). Define
\begin{equation}
\label{eqn:lq110}
 L_{n, p} = K(X_n) - \sum_{r=0}^p \kappa_r U_{n,r},
\end{equation}
where we recall $\kappa_r = K^{(r)}_\infty(0)$, and $U_{n,r}$ is
the Volterra process
\begin{equation}
\label{eqn:lq11}
 U_{n, r} = \sum_{0 \le j_1 < \ldots < j_r < \infty}
 \prod_{s=1}^r a_{j_s} \varepsilon_{n-j_s}.
\end{equation}
We can view $L_{n, p}$ as the remainder of the $p$-th order
Volterra expansion of $K(X_n)$. Note that $\kappa_r = 0$ if $1\le
r < p$. In the special case of Gaussian processes, $L_{n, p}$ is
closely related to the Hermite expansion. In Lemma
\ref{lem:prdmcexp} we compute the predictive dependence measures
for the Volterra process $U_{n, r}$ and for $Y_n = K(X_n)$.

\begin{lemma}
\label{lem:prdmcexp} Let $\nu \ge 2$, $r\geq1$ and assume $\varepsilon_i \in
{\cal L}^\nu$. Let $A_n = \sum_{j=n}^\infty a_j^2$. Then
\begin{eqnarray}
\label{eqn:prdmV}
 \|{\cal P}_0 U_{n, r} \|^2_\nu = O(a_n^2 A_n^{r-1}).
\end{eqnarray}
\end{lemma}

\begin{proof}
Let $D_i, i \in \Z$, be a sequence of martingale
differences with $D_i \in {\cal L}^\nu$. By the Burkholder and the
Minkowski inequalities, there exists a constant $C_\nu$ which only
depends on $\nu$ such that, for all $m\ge 1$, we have
\begin{eqnarray}
\label{eqn:DMIM} \|D_1 + \ldots + D_m \|^2_\nu \le C_\nu
 (\|D_1 \|^2_\nu + \ldots + \| D_m \|^2_\nu).
\end{eqnarray}
We now apply the induction argument and show that, for all $r \ge
1$,
\begin{eqnarray}\label{eqn:Unr}
 \|\E (U_{n, r} | {\cal F}_0) \|^2_\nu = O(A_n^{r}).
\end{eqnarray}
Clearly (\ref{eqn:Unr}) holds with $r = 1$. By (\ref{eqn:DMIM}),
\begin{eqnarray*}
 \|\E (U_{n, r+1} | {\cal F}_0) \|^2_\nu
 &=&  \left\| \sum_{i_1 = -\infty}^0
  a_{n-i_1} \varepsilon_{i_1}
  \sum_{i_{r+1} < \ldots < i_2 < i_1}
  \prod_{k=2}^{r+1} a_{n-i_k} \varepsilon_{i_k} \right\|_\nu^2 \cr
 &\le& C_\nu \sum_{i_1 = -\infty}^0
        a^2_{n-i_1} \| \varepsilon_{i_1} \|_\nu^2
     \left\| \sum_{i_{r+1} < \ldots < i_2 < i_1}
  \prod_{k=2}^{r+1} a_{n-i_k} \varepsilon_{i_k} \right\|_\nu^2 \cr
 &=& C_\nu \| \varepsilon_0 \|_\nu^2
 \sum_{i_1 = -\infty}^{-1}
        a^2_{n-i_1} \|\E (U_{n, r} | {\cal F}_{i_1}) \|^2_\nu.
\end{eqnarray*}
By stationarity, $\|\E (U_{n, r} | {\cal F}_{i_1}) \|^2_\nu = \|\E
(U_{n-i_1, r} | {\cal F}_0) \|^2_\nu$. Then, by the induction hypothesis,
\begin{eqnarray*}
 \|\E (U_{n, r+1} | {\cal F}_0) \|^2_\nu
 \le C_\nu \| \varepsilon_0 \|_\nu^2
 \sum_{i_1 = -\infty}^0 a^2_{n-i_1} O(A_{n-i_1}^r)
 = O(A_n^{r+1}).
\end{eqnarray*}
Hence (\ref{eqn:Unr}) holds for all $r \ge 1$. By independence,
${\cal P}_0 U_{n, r} = a_n \varepsilon_0 \E(U_{n, r-1} | {\cal
F}_{-1})$, which implies (\ref{eqn:prdmV}) by (\ref{eqn:Unr}).
\end{proof}

\begin{lemma}
\label{lem:S096311} Assume  $r\in\N$, $r (2\beta-1) < 1$, and
$\varepsilon_i \in {\cal L}^\nu$, $\nu \ge 2$. Let $T_{n, r} =
\sum_{i=1}^n U_{i, r}$.
\begin{enumerate}
\renewcommand{\labelenumi}{(\roman{enumi})}
\item Let $(c_i)_{i\in\N}$ be a real valued sequence. Then
\begin{equation*}
\left\|\sum_{i=1}^{n}c_iU_{i,r}\right\|_\nu^2
 =O\left(n^{1-r(2\beta-1)}\ell^{2r}(n)
 \sum_{i=1}^{n}c_i^2\right).
\end{equation*}
\item Assume that  $n \le N$ and
$\varphi\in(0,\beta-1/2)$. Then
\begin{equation*}
 { {\|T_{n, r} - \E(T_{n, r} | {\cal F}_{-N}^\infty) \|_\nu}
 \over {n^{1- r(\beta-1/2)} \ell^r(n)}}
 = O[(n/N)^\varphi].
\end{equation*}
\item If additionally Condition \ref{cond:1} holds, we have for
some $\varphi_2 > 0$
\begin{eqnarray}
 \label{eqn:S295381}
 { {\|T_{n, r} \|^2} \over {n^{2- r(2\beta -1)} c^{2r}
  \left\|Z_{r,\beta}(1)\right\|^2
  \left\|\varepsilon_1\right\|^{2r}}  }
 = 1 + O( n^{-\varphi_2}).
\end{eqnarray}
\end{enumerate}
\end{lemma}

\begin{proof}
For (i), we use the following decomposition with the help of the
projection operator
\begin{equation*}
\sum_{i=1}^{n}c_iU_{i,r}=\sum_{j=1}^{n}\sum_{l=0}^\infty
 \sum_{i=1}^{n}{\cal P}_{-ln-j+i}(c_iU_{i,r}).
\end{equation*}
Note that both $\{\sum_{i=1}^{n}{\cal P}_{-ln-j+i}(c_iU_{i,r})
\}_{l \in \mathbb{N}}$ and $\{{\cal P}_{-ln-j+i}(U_{1,r})
\}_{i=1}^n$ form martingale differences, for any $j \in
\{1,\ldots,n\}$, we have
\begin{eqnarray*}
\left\|\sum_{l=0}^\infty\sum_{i=1}^{n} {\cal P}_{-ln-j+i}(c_iU_{i,r})\right\|_\nu^2
 & \leq & C\sum_{l=0}^\infty\left\|\sum_{i=1}^{n} {\cal P}_{-ln-j+i}(c_iU_{i,r})\right\|_\nu^2 \cr
 & \leq & C\sum_{l=0}^\infty\sum_{i=1}^{n}c_i^2 \left\|{\cal P}_{0}(U_{ln+j,r})\right\|_\nu^2.
\end{eqnarray*}
Hence by Lemma \ref{lem:prdmcexp}, we have
\begin{eqnarray*}
\left\|\sum_{l=0}^\infty\sum_{i=1}^{n}
 {\cal P}_{-ln-j+i}(c_iU_{i,r})\right\|_\nu^2
 & \leq & C\sum_{i=1}^nc_i^2\left(a_j^2A_j^{r-1}
  +\sum_{l=1}^\infty a_{j+ln}^2A_{j+ln}^{r-1}\right)\\
 & = & \sum_{i=1}^n c_i^2
 O \left(j^{-2\beta-(r-1)(2\beta-1)}\ell^{2r}(j)\right).
\end{eqnarray*}
Then by the triangle inequality we can get
\begin{eqnarray*}
\left\|\sum_{i=1}^{n}c_iU_{i,r}\right\|_\nu^2
 & = & \left\|\sum_{j=1}^{n}
  \sum_{l=0}^\infty\sum_{i=1}^{n}
  {\cal P}_{-ln-j+i}(c_iU_{i,r})\right\|_\nu^2\\
 & \leq & C\sum_{i=1}^n c_i^2 \Big(\sum_{j=1}^n
  j^{-\beta-(r-1)(\beta-\frac{1}{2})}\ell^r(j)\Big)^2 \cr
  &=& O \Big(n^{1-r(2\beta-1)}\ell^{2r}(n)
  \sum_{i=1}^{n}c_i^2 \Big).
\end{eqnarray*}

For (ii), we define the future projection operator ${\cal Q}_j := \E(\cdot|{\cal F}_{j}^\infty) - \E(\cdot|{\cal F}_{j+1}^\infty)$ and obtain
\begin{equation*}
T_{n,r}-E(T_{n,r}|{\cal F}_{-N}^\infty) = \sum_{j=N+1}^\infty {\cal Q}_{-j}(T_{n,r}).
\end{equation*}
Note that ${\cal Q}_j (T_{n,r}) = \sum_{i=1}^{n} a_{i-j}
\varepsilon_j \E(U_{i,r-1}|{\cal F}_{j+1}^\infty)$ which forms a
sequence of martingale differences for $j \in \mathbb{Z}$, we have
\begin{eqnarray*}
\left\|T_{n,r}-\E(T_{n,r}|{\cal F}_{-N}^\infty)\right\|_\nu^2
 &\leq& C \sum_{j=N}^\infty \left\|a_{i+j}\varepsilon_{-j}
  \E(U_{i,r-1}|{\cal F}_{-j+1}^\infty)\right\|_{\nu}^2 \cr
 &\leq& C\sum_{j=N}^\infty \left\|a_{i+j} U_{i,r-1}\right\|_{\nu}^2.
\end{eqnarray*}
Hence by using part (i) of this lemma, we have
\begin{eqnarray*}
\left\|T_{n,r}-\E(T_{n,r}|{\cal F}_{-N}^\infty)\right\|_\nu^2
 & \leq & C \sum_{j=N}^\infty \sum_{i=1}^{n}
  a^2_{i+j}n^{1-(r-1)(2\beta-1)}\ell^{2(r-1)}(n)\\
 & \leq & C N^{-(2\beta-1)}\ell^2(N)
 n^{2-(r-1)(2\beta-1)} \ell^{2(r-1)}(n).
\end{eqnarray*}
Therefore by the the slow variation of  $\ell(\cdot)$, we have for
some $\varphi \in (0, \beta - 1/2)$,
\begin{equation*}
 { {\|T_{n, r} - \E(T_{n, r} | {\cal F}_{-N}^\infty) \|_\nu}
  \over {n^{1- r(\beta-1/2)} \ell^r(n)}}
  = O \left(\frac{n^{\beta-\frac{1}{2}}\ell(N)}
  {N^{\beta-\frac{1}{2}}\ell(n)}\right)
  = O ((n/N)^\varphi).
\end{equation*}

We now prove (iii). Without loss of generality let the constant
$c$ in Condition \ref{cond:1} be $1$ and assume $\| \varepsilon_1
\| = 1$. For $\beta \in (1/2, \, 1/2 + 1/(2r))$, define $a_{i,
\beta} = i^{-\beta}$ if $i \ge 1$, $a_{i, \beta} = 1$ if $i = 0$
and $a_{i, \beta} = 0$ if $i < 0$. Let $\beta_k \in (1/2, \, 1/2 +
1/(2r))$, $1\le k \le r$, and define
\begin{equation*}
 T_{n, \beta_1, \ldots, \beta_r} =
  \sum_{j_r < \ldots < j_1 \le n}
  \sum_{i=1}^n \prod_{k=1}^r
  a_{i-j_k, \beta_k} \varepsilon_{j_k}.
\end{equation*}
Using the approximations that, for $1/2 < \beta < 1$,
$\sum_{i=1}^n i^{-\beta} = n^{1-\beta} / (1-\beta) + O(1)$ and
$\sum_{i=n_1}^{n_2} i^{-\beta} = (n_2^{1-\beta} - n_1^{1-\beta}) /
(1-\beta) + O(n_2^{-\beta} + n_1^{-\beta})$ when $n_1, n_2 \ge 1$,
by elementary but tedious calculations, we have for some
$\varphi_3 > 0$ that
\begin{eqnarray}
 \label{eqn:D21625}
 { {\|T_{n, \beta_1, \ldots, \beta_r} \|^2}
 \over {n^{2- \sum_{k=1}^r(2\beta_k -1)}
  \zeta_{\beta_1, \ldots, \beta_r} }  }
 = 1 + O( n^{-\varphi_3}),
\end{eqnarray}
where, recall that ${\cal S}_t = \{(u_1, \ldots, u_r) \in \R^r: \,
-\infty < u_1 < \ldots <u_r <t\}$,
\begin{equation*}
\zeta_{\beta_1, \ldots, \beta_r} = \int_{{\cal S}_1} \left[\int_0^1 \prod_{k=1}^r g_{\beta_k}(v-u_k) d v \right]^2 d u_1 \ldots d u_r.
\end{equation*}
Note that $\|Z_{r,\beta}(1)\|^2 = \zeta_{\beta, \ldots, \beta}$ if
$1/2 < \beta < 1/2 + 1/(2 r)$. We now show that (\ref{eqn:D21625})
implies (\ref{eqn:S295381}). To this end, for notational clarity,
we only consider $r = 2$. The general case similarly follows. Let
$\phi > 0$ be such that $\phi < 1/2 + 1/(2 r) - \beta$, hence
$\phi + \beta < 1/2 + 1/(2 r)$. Writing
\begin{eqnarray}\label{eqn:D21715}
 a_{i-j_1} a_{i-j_2} - a_{i-j_1, \beta} a_{i-j_2, \beta}
  = (a_{i-j_1} - a_{i-j_1, \beta} )a_{i-j_2}
  + a_{i-j_1, \beta} (a_{i-j_2} - a_{i-j_2, \beta}).
\end{eqnarray}
By Condition \ref{cond:1}, for $j_1 \ge 1$, $a_{j_1} - a_{j_1,
\beta} = O(j_1^{-\beta-\phi})$. Hence $a_{j} - a_{j, \beta} =
O(a_{j, \beta+\phi})$ for $j \in \Z$. Applying (\ref{eqn:D21625})
to the case with $\beta_1 = \beta$ and $\beta_2 = \phi + \beta$,
we have
\begin{eqnarray*}
\left\| \sum_{j_2 < j_1 \le n}
  \sum_{i=1}^n (a_{i-j_1} - a_{i-j_1, \beta} )a_{i-j_2}
  \varepsilon_{j_1} \varepsilon_{j_2} \right\|^2
  &=& \sum_{j_2 < j_1 \le n}
  \left[\sum_{i=1}^n (a_{i-j_1} - a_{i-j_1, \beta} )a_{i-j_2}
  \right]^2\cr
  &=& O(1)
  \left\| \sum_{j_2 < j_1 \le n}
  \sum_{i=1}^n a_{i-j_1, \beta_2} a_{i-j_2, \beta_1}
  \varepsilon_{j_1} \varepsilon_{j_2} \right\|^2 \cr
  &=& O(n^{2-(2\beta_1-1)-(2\beta_2-1)}).
\end{eqnarray*}
A similar relation can be obtained by replacing $(a_{i-j_1} -
a_{i-j_1, \beta} )a_{i-j_2}$ in the preceding equation by
$a_{i-j_1, \beta} (a_{i-j_2} - a_{i-j_2, \beta})$. Hence, by
(\ref{eqn:D21715}), $ \|T_{n, 2} - T_{n, \beta, \beta}\|^2 =
O(n^{2-(2\beta_1-1)-(2\beta_2-1)}) $, which by (\ref{eqn:D21625})
implies (\ref{eqn:S295381}) since $\beta_1 = \beta$ and $\beta_2 =
\phi + \beta$.
\end{proof}

\begin{lemma}
\label{lem:S6511} Assume Condition \ref{cond:2} holds with $\nu
\ge 2$ and $K$ has power rank $p \ge 1$ with respect to the
distribution of $X_i$ such that $r (2\beta-1) < 1$. Then we have: (i) $\|S_n \|_\nu = O(\sigma_{n, p})$; and (ii) there exists
$\varphi_1, \varphi_2 > 0$ such that
\begin{equation*}
{\|S_n - \E(S_n | {\cal F}_{-N}^\infty) \|_\nu \over \sigma_{n,p}} = O(n^{-\varphi_1}) + O[(n/N)^{\varphi_2}].
\end{equation*}
\end{lemma}

\begin{proof}
Recall Lemma \ref{lem:S096311} for $T_{n, p}$.
Observe that $Y_n = L_{n, p} + \kappa_p U_{n, p}$ and $S_n =
S_n(L^{(p)}) + \kappa_p T_{n, p}$. Since
\begin{eqnarray*}
\|S_n - \E(S_n | {\cal F}_{-N}^\infty) \|_\nu
 &\le& \|S_n(L^{(p)})
  - \E(S_n(L^{(p)}) | {\cal F}_{-N}^\infty) \|_\nu + \|\kappa_p T_{n, p}
 - \E(\kappa_p T_{n, p} | {\cal F}_{-N}^\infty) \|_\nu \cr
 &\le& 2 \|S_n(L^{(p)}) \|_\nu
 + |\kappa_p|\|T_{n, p}
 - \E(T_{n, p} | {\cal F}_{-N}^\infty) \|_\nu,
\end{eqnarray*}
by Lemma \ref{lem:S096311}, it suffices to show that
\begin{eqnarray}\label{eqn:S7311}
  \frac{\|S_n(L^{(p)}) \|_\nu}{\sigma_{n,p}} = O(n^{-\varphi_1}).
\end{eqnarray}
By the argument of Theorem 5 in Wu (2006), Condition \ref{cond:2}
with $\nu \ge 2$ implies that
\begin{eqnarray}
\label{eqn:54401}
 \|{\cal P}_0 L_{n, p} \|^2_\nu
  = a_n^2 O(a_n^2 + A_{n+1}(4) + A_{n+1}^p),
\end{eqnarray}
where $A_n(4) = \sum_{t=n}^\infty a_t^4$ and $A_n =
\sum_{t=n}^\infty a_t^2$. Let $\theta_i = |a_i| [ |a_i| +
A^{1/2}_{i+1}(4) + A_{i+1}^{p/2} ]$ if $i \ge 0$ and $\theta_i =
0$ if $i < 0$ (Theorem 5 and Lemma 2 in Wu consider only the case
$\nu=2$, but the case $\nu>2$ can be proved analogously using the
Burkholder inequality). Write $\Theta_n = \sum_{k=0}^n \theta_k$
and $\Xi_{n, p} = n \Theta_{n}^2 + \sum_{i=1}^\infty (\Theta_{n+i}
- \Theta_i)^2$. By (\ref{eqn:DMIM}), since ${\cal P}_k
S_n(L^{(p)})$, $k = -\infty, \ldots, n-1, n$, for martingale
differences, we have
\begin{eqnarray*}
\|S_n(L^{(p)}) \|^2_\nu
 &\leq& C_\nu \sum_{k=-\infty}^n \|{\cal P}_k S_n(L^{(2)}) \|^2_\nu \cr
 &\leq& C_\nu \sum_{k=-\infty}^n \left( \sum_{i=1}^n \theta_{i-k} \right)^2\cr
 &=& O(\Xi_{n, p}).
\end{eqnarray*}
By (i), (ii) and (iii) of Corollary 1 in Wu (2006), we have
$\Xi^{1/2}_{n, p} / \sigma_{n, p} = O(n^{1/2-\beta} \ell(n))$ if
$(p+1) (2\beta-1) < 1$ and $\Xi^{1/2}_{n, p} / \sigma_{n, p} =
O(n^{p(\beta-1/2)-1/2} \ell_0(n))$ if $(p+1) (2\beta-1) \ge 1$.
Here $\ell_0$ is a slowly varying function. Note that both $1/2 -
\beta$ and ${p(\beta-1/2)-1/2}$ are negative, (\ref{eqn:S7311}) follows.
\end{proof}

\section*{References}
\par\noindent\hangindent2.3em\hangafter 1
Avram, F. and Taqqu, M.~S. (1987) Noncentral limit
theorems and appell polynomials. {\it Ann. Prob.} {\bf 15},
767--775.

\par\noindent\hangindent2.3em\hangafter 1
Bertail, P., Politis, D.~N. and Romano, J.~P. (1999)
On subsampling estimators with unknown rate of convergence. {\it
J. Amer. Statist. Assoc.} {\bf 94}, 569--579.

\par\noindent\hangindent2.3em\hangafter 1
Billingsley, P. (1968) Convergence of Probability measures, Wiley,
New York.

\par\noindent\hangindent2.3em\hangafter 1
Bingham, N.~H., Goldie, C.~M., and Teugels, J.~M. (1987) Regular
variation, Cambridge University Press, Cambridge, UK.

\par\noindent\hangindent2.3em\hangafter 1
Chow, Y.~S. and Teicher, H. (1997) {\it Probability theory.
Independence, interchangeability, martingales.} Third edition.
Springer, New York.

\par\noindent\hangindent2.3em\hangafter 1
Chung, C.~F. (2002) Sample means, sample autocovariances, and
linear regression of stationary multivariate long memory
processes. {\it Econom. Theory} {\bf 18}, 51--78.

\par\noindent\hangindent2.3em\hangafter 1
Davydov, Y.~A. (1970) The invariance principle for stationary
processes, {\it Theory of Probability and its Applications} {\bf
15}, 487--498.

\par\noindent\hangindent2.3em\hangafter 1
Dittmann, I. and Granger, C.~W.~J. (2002) Properties of nonlinear
transformations of fractionally integrated processes. {\it Journal
of Econometrics} {\bf 110}, 113--133.

\par\noindent\hangindent2.3em\hangafter 1
Dobrushin, R.~L. and Major, P. (1979) Non-central limit theorems
for nonlinear functionals of Gaussian fields. {\it Z. Wahrsch.
Verw. Gebiete} {\bf 50}, 27--52.

\par\noindent\hangindent2.3em\hangafter 1
Giraitis, L., Robinson, P.~M. and Surgailis, D. (1999)
Variance-type estimation of long memory. {\it Stochastic processes
and their applications} {\bf 80}, 1--24.

\par\noindent\hangindent2.3em\hangafter 1
Granger, C.~W.~J. and Joyeux, R. (1980) An introduction to long-memory time series models and fractional differencing. {\it J.
Time Ser. Anal.} {\bf 1}, 15--29.

\par\noindent\hangindent2.3em\hangafter 1
Hall, P., Horowitz, J.~L. and Jing, B.-Y. (1995) On blocking
rules for the bootstrap with dependent data. {\it Biometrika} {\bf
82}, 561--574.

\par\noindent\hangindent2.3em\hangafter 1
Hall, P., Jing, B.-Y. and Lahiri, S.~N. (1998) On the
sampling window method for long-range dependent data. {\it
Statist. Sinica} {\bf 8}, 1189--1204.

\par\noindent\hangindent2.3em\hangafter 1
Ho, H.-C. and Hsing, T. (1997). Limit theorems for functionals of
moving averages, {\it Annals of Probability} {\bf 25}, 1636--1669.

\par\noindent\hangindent2.3em\hangafter 1
Hosking, R.~J. (1981) Fractional differencing. {\it Biometrika} {\bf 68}, 165--176.

\par\noindent\hangindent2.3em\hangafter 1
Hurvich, C.~M., Moulines, E. and Soulier, P. (2005) Estimating
long memory in volatility. {\it Econometrica} {\bf 73},
1283--1328.

\par\noindent\hangindent2.3em\hangafter 1
K\"unsch, H.~R. (1989) The jackknife and the bootstrap for
general stationary observations. {\it Ann. Statist.} {\bf 17},
1217--1241.

\par\noindent\hangindent2.3em\hangafter 1
Lahiri, S.~N. (1993) On the moving block bootstrap under long
range dependence. {\it Statist. Probab. Lett.} {\bf 18}, 405--413.

\par\noindent\hangindent2.3em\hangafter 1
Moulines, E. and Soulier, P. (1999) Broad-band log-periodogram
regression of time series with long-range dependence. {\it Ann.
Statist.} {\bf 27}, 1415--1439.

\par\noindent\hangindent2.3em\hangafter 1
Nordman, D.~J. and Lahiri, S.~N. (2005) Validity of the
sampling window method for long-range dependent linear processes.
{\it Econometric Theory} {\bf 21}, 1087--1111.

\par\noindent\hangindent2.3em\hangafter 1
Politis, D.~N., Romano, J.~P. and Wolf, M. (1999) {\it
Subsampling}. Springer, New York.

\par\noindent\hangindent2.3em\hangafter 1
Robinson, P.~M. (1994) Semiparametric analysis of long-memory time
series. {\it Ann. Statist.} {\bf 22}, 515--539.

\par\noindent\hangindent2.3em\hangafter 1
Robinson, P.~M. (1995a) Log-periodogram regression of time series
with long range dependence. {\it Ann. Statist.} {\bf 23},
1048--1072.

\par\noindent\hangindent2.3em\hangafter 1
Robinson, P.~M. (1995b) Gaussian semiparametric estimation of long
range dependence. {\it Ann. Statist.} {\bf 23}, 1630--1661.

\par\noindent\hangindent2.3em\hangafter 1
Surgailis, D. (1982) Domains of attraction of self-similar
multiple integrals. {\it Litovsk. Mat. Sb.} {\bf 22}, 185--201.

\par\noindent\hangindent2.3em\hangafter 1
Taqqu, M.~S. (1975) Weak convergence to fractional Brownian
motion and to the Rosenblatt process. {\it Z. Wahrsch. Verw.
Gebiete} {\bf 31}, 53-83.

\par\noindent\hangindent2.3em\hangafter 1
Teverovsky, V. and Taqqu, M.~S. (1997) Testing for long-range
dependence in the presence of shifting mean or a slowly declining
trend, using a variance-type estimator. {\it J. Time Ser. Anal.}
{\bf 18}, 279--304.

\par\noindent\hangindent2.3em\hangafter 1
Wu, W.~B. (2005) Nonlinear system theory: Another look at
dependence, {\it Proc. Nat. Journal Acad. Sci. USA} {\bf 102},
14150--14154.

\par\noindent\hangindent2.3em\hangafter 1
Wu, W.~B. (2006) Unit root testing for functionals of linear
processes. {\it Econometric Theory} {\bf 22}, 1--14

\par\noindent\hangindent2.3em\hangafter 1
Wu, W.~B. (2007) Strong invariance principles for dependent
random variables. {\it Ann. Probab.} {\bf 35}, 2294--2320.
\end{document}